\documentclass[amstex,12pt]{article}

\usepackage[margin=2cm]{geometry}
\usepackage{amsmath, amssymb, amsthm}
\usepackage{mathrsfs}

\usepackage{mathtools}
\mathtoolsset{showonlyrefs=true}

\usepackage{tikz}
\usepackage{tikz-cd}
\usepackage[all]{xy}
\usetikzlibrary{patterns}
\usetikzlibrary{backgrounds}

\newtheorem{theorem}{Theorem}[section]
\newtheorem{lemma}[theorem]{Lemma}
\newtheorem{proposition}[theorem]{Proposition}
\newtheorem{corollary}[theorem]{Corollary}

\theoremstyle{definition}

\title{Arithmetic Y-frieze patterns of width 3 and 4}
\author{Katsuhiko Matsuzaki and Taiki Resnick
\thanks{Department of Mathematics, School of Education, Waseda University,
Shinjuku, Tokyo 169-8050, Japan}}

\date{}

\begin{document}

\maketitle

\begin{abstract}
We determine all arithmetic Y-Frieze patterns of width $3$ and $4$.
As a consequence, for $n=3,4$, we verify the surjectivity of a map $p_n$ 
which corresponds arithmetic Y-Frieze patterns of width $n$ to Coxeter's Frieze patterns.
\end{abstract}

\section{Introduction}

In his work, de St. Germain \cite{dSG} introduces Y-frieze patterns as a variant of classical frieze patterns by Coxeter \cite{Cox}, 
motivated by the Y-systems of Zamolodchikov and their connection to cluster algebras. 
These patterns consist of staggered infinite rows of rational numbers satisfying a recurrence 
relation that mirrors the structure found in both classical frieze patterns and Y-systems of type $A_n$. 
As in the original theory, Y-frieze patterns exhibit a form of periodicity governed by a glide symmetry, 
and a subclass of arithmetic closed Y-frieze patterns—those with positive integer entries and a finite number of rows—is considered.

In \cite{dSG0,dSG}, it is proved that for fixed width $n$, 
the number of arithmetic Y-frieze patterns is finite 
and pose the natural enumeration problem of determining all such patterns. 
Moreover, a conjecture is arose asserting that a natural map $p_n$, relating Coxeter frieze patterns to Y-frieze patterns, is surjective for all $n$. This conjecture is clearly true for $n = 1$ and $n = 2$.

In the present paper, we solve the cases $n = 3$ and $n = 4$ by explicitly determining all arithmetic Y-frieze patterns of width $3$ and $4$ and proving that the associated maps $p_3$ and $p_4$ are indeed surjective. This confirms the conjecture in this case and provides a complete classification of the patterns involved. 
Our approach combines direct enumeration by computer algorithm with a structural analysis of the recurrence relation and its constraints 
on positive integer solutions.

\section{Arithmetic Y-frieze patterns of width $3$}

A {\em Y-frieze pattern} is a
collection of staggered infinite rows of rational numbers subject to
the Y-diamond rule:
$$
WE = (1 + N)(1 + S)\quad  {\rm for\ every}\quad 
\begin{array}{rrr}
&N& \\
W&&E \ . \\
&S& \\
\end{array}
$$ 
The initial row of a Y-frieze pattern consists
of all $0$ entries. This is the $0$ th row and
subsequent rows are the first row, and the second row, etc.
If the entries of the $(n+1)$ th row are all $0$ for some $n \geq 1$ 
and there are no such rows between $0$ th and $(n+1)$ th rows,
we call it a {\em closed} Y-frieze pattern of {\em width} $n$. 

In is proved by \cite[Theorem 5.2]{dSG} that
every closed Y-frieze pattern admits a glide symmetry
as Coxeter frieze patterns do. This symmetry is the composition of a horizontal
reflection and a translation. A fundamental domain of the glide symmetry
for a Y-frieze pattern of width $3$ is given as follows.

\begin{center}
\begin{tikzpicture}[x=1.2cm, y=0.7cm, font=\small]
\def\rows{
  {$0$, $0$, $0$, $0$, $0$},
  {$a$, $d$, $g$, $i$,  },
  {$b$, $e$, $h$,  ,  },
  {$c$, $f$,  ,  ,  },
  {$0$,  ,  ,  ,  },
}
\foreach \row [count=\i from 0] in \rows {
  \foreach \val [count=\j from 0] in \row {
    \ifx\val\empty
    \else
      \node at ({\j + 0.5*\i}, {-\i}) {\val};
    \fi
  }
}
\end{tikzpicture}
\end{center}

The entries in the above Y-frieze patters of width $3$ 
satisfy the following $6$ equations by the diamond relation
$WE=(N+1)(S+1)$:
\begin{align}\label{6equations}
&ad=1+b, \quad be=(1+d)(1+c), \quad cf=1+e, \\
&fg=1+e, \quad eh=(1+g)(1+f), \quad gi=1+h. \\
\end{align}

Solving these equations by given $a$, $b$, and $c$, we represent the other $6$ entries as follows:
\begin{itemize}
\item $d = \frac{b + 1}{a}$,
\item $e = \frac{(c+1)(a+b+1)}{a b}$,
\item $f = \frac{a b + a c + b c +a  + b + c + 1}{a b c}$,
\item $g = \frac{a b + a c + b c +a  + b + c + 1}{b(b+1)}$,
\item $h = \frac{(a+1)(b+c+1)}{b c}$,
\item $i = \frac{b + 1}{c}$.
\end{itemize}

A frieze pattern is called {\em arithmetic} if all non-zero entries are positive integers.
It is proved in \cite[Theorem 5.5]{dSG} that for each $n \geq 1$, the number of arithmetic Y-frieze patterns of width $n$ is finite.
We determine all arithmetic Y-frieze patterns of width $3$.
Since $d,e,f,g,h,i$ are all positive integers in this case, 
the numerator in each fraction above is more than or equal to its denominator.
Then, we obtain the following $6$ inequalities, which are necessary conditions for arithmeticity:
\begin{itemize}
\item[(i)] $b + 1 \geq a$,
\item[(ii)] $(c+1)(a+b+1) \geq ab$,
\item[(iii)] $ab + ac + bc +a  + b + c + 1 \geq abc$,
\item[(iv)] $ab + ac + bc +a  + b + c + 1 \geq b(b+1)$,
\item[(v)] $(a+1)(b+c+1) \geq bc$,
\item[(vi)] $b + 1 \geq c$.
\end{itemize}

By using these 6 inequalities, we prove the following claims for
any arithmetic Y-frieze pattern of width $3$. 

\begin{proposition}\label{2.1}
If $a \geq 5$ and $c \geq 5$, then inequality {\rm (iii)}
is not satisfied.
\end{proposition}

\begin{proof}
Inequality (iii) is equivalent to
$$
\frac{(a+1)(b+1)(c+1)}{abc} \geq 2.
$$
If $a \geq 5$ and $c \geq 5$, then $b \geq 4$ by (i) or (vi).
In this case, the maximum of possible values of the left side term in the above inequality is
$9/5$. Hence, the inequality is not satisfied.
\end{proof}

\begin{proposition}\label{2.2}
If $a \leq 4$, then $c \leq 11$ and $b \leq 18$.
\end{proposition}

\begin{proof}
If $a \leq 4$, then inequality (v) implies 
$5b+5c+5 \geq bc$, which is equivalent to
$$
\frac{(b+1)(c+1)}{bc} \geq \frac{6}{5}.
$$ 
If $c \geq 12$, then $b \geq 11$ by (vi), but these values do not satisfy 
this inequality. Hence, we have $c \leq 11$.

Moreover, having $a \leq 4$ and $c \leq 11$, inequality (iv) implies that
$$
b^2 -15b -60 \leq 0.
$$
This holds only when $b \leq 18$.
\end{proof}

In the last proposition, the roles of $a$ and $c$ can be exchanged because
inequalities (i)--(vi) are symmetric with respect to $a$ and $c$.
Then, these necessary conditions derive the following numerical bounds for the entries $a$, $b$, and $c$.

\begin{lemma}
An arithmetic Y-frieze of width $3$ exists only when either
$a \leq 4$, $c \leq 11$, and $b \leq 18$ or
$a \leq 11$, $c \leq 4$, and $b \leq 18$.
\end{lemma}

Having established bounds for $(a,b,c)$, we can determine when $d,e,f,g,h,i$ are positive integers by numerical 
experiments—checking the equations a finite number of times—and thereby obtain all integer solutions to \eqref{6equations}. This resolves a problem posed in \cite[p.75]{dSG0}.

\begin{theorem}\label{10(abc)}
The triple $(a,b,c)$ is the first diagonal of an arithmetic Y-frieze pattern of width $3$ exactly when
$(a,b,c)$ is in the following list:
$$
(1, 1, 2), (1, 2, 3), (1, 4, 5), (2, 1, 1), (2, 3, 2), (2, 9, 5), (3, 2, 1), (3, 8, 3), (5, 4, 1), (5, 9, 2).
$$
\end{theorem}

A (Coxeter) frieze pattern is a collection of staggered infinite rows arranged so as to satisfy the unimodular rule $WE-NS=1$. Its initial row consists entirely of $0$ entries, and the next row consists entirely of $1$ entries. A frieze pattern is said to be closed if it has a row of $1$ followed by a row of $0$; in this case, the number $n \geq 1$ of rows strictly between the first two rows of $1$ is called its {width. A frieze pattern is said to be arithmetic if all of its rows, except for the rows of $0$, consist entirely of positive integers.

Let ${\rm Frieze}(n)$ denote the set of all arithmetic frieze patterns of width $n$, and let ${\rm YFrieze}(n)$ denote the set of all arithmetic Y-frieze patterns of width $n$. In \cite[Theorem 2.8]{dSG}, it is proved that there exists a well-defined map
$$
p_n : {\rm Frieze}(n) \to {\rm YFrieze}(n)
$$
for each $n \geq 1$, under which the second row of an element of ${\rm Frieze}(n)$ coincides with the first row of the corresponding element of ${\rm YFrieze}(n)$. In \cite[Conjecture 10]{dSG0}, it is conjectured that the map $p_n$ is surjective for all $n \geq 1$. As a consequence of Theorem \ref{10(abc)}, we confirm that this conjecture holds for $n=3$.

The cardinality $|{\rm Frieze}(3)|$ of the arithmetic (Conway) frieze patterns of width $3$ is equal to the Catalan number $C_{3+1} = 14$, which is also the number of triangulations of a convex $(3+3)$-gon. Theorem \ref{10(abc)} implies that $|{\rm YFrieze}(3)|$, for the arithmetic Y-frieze patterns of width $3$, is $10$. It is noted in \cite{dSG0} that \cite[Lemma 7.5]{CH} implies that the map $p_n$ is at most $2$-to-$1$ when $n \geq 1$ is odd.

Here we present the explicit correspondence between these frieze patterns under the map $p_3$. The numbers $s:t$ (with $s/t$ equal to either $1$ or $2$) written under the arrow
$\xlongrightarrow{p_3}$
indicate that the frieze pattern on the left represents $s$ distinct patterns by cyclic permutation, whereas the Y-frieze pattern on the right represents $t$ distinct patterns by cyclic permutation.

\bigskip
\begin{tikzpicture}[x=1.0cm, y=0.6cm, font=\small]
\def\rows{
 {$1$, $1$, $1$, $1$, $1$, $1$, , , ,$0$, $0$, $0$, $0$, $0$, $0$},
 {$2$, $1$, $3$, $2$, $1$, $3$, , , ,$1$, $2$, $5$, $1$, $2$, $5$},
 {$1$, $2$, $5$, $1$, $2$, $5$, ,$\xlongrightarrow[3:3]{p_3}$, ,$1$, $9$, $4$, $1$, $9$, $4$},
 {$1$, $3$, $2$, $1$, $3$, $2$, , , ,$2$, $5$, $1$, $2$, $5$, $1$},
 {$1$, $1$, $1$, $1$, $1$, $1$, , , ,$0$, $0$, $0$, $0$, $0$, $0$},
}
\foreach \row [count=\i from 0] in \rows {
\foreach \val [count=\j from 0] in \row {
\ifx\val\empty
\else
\node at ({\j + 0.5*\i}, {-\i}) {\val};
\fi
}}
\end{tikzpicture}

\bigskip

\begin{tikzpicture}[x=1.0cm, y=0.6cm, font=\small]
\def\rows{
 {$1$, $1$, $1$, $1$, $1$, $1$, , , ,$0$, $0$, $0$, $0$, $0$, $0$},
 {$2$, $3$, $1$, $2$, $3$, $1$, , , ,$5$, $2$, $1$, $5$, $2$, $1$},
 {$5$, $2$, $1$, $5$, $2$, $1$, ,$\xlongrightarrow[3:3]{p_3}$, ,$9$, $1$, $4$, $9$, $1$, $4$},
 {$3$, $1$, $2$, $3$, $1$, $2$, , , ,$5$, $2$, $1$, $5$, $2$, $1$},
 {$1$, $1$, $1$, $1$, $1$, $1$, , , ,$0$, $0$, $0$, $0$, $0$, $0$},
}
\foreach \row [count=\i from 0] in \rows {
\foreach \val [count=\j from 0] in \row {
\ifx\val\empty
\else
\node at ({\j + 0.5*\i}, {-\i}) {\val};
\fi
}}
\end{tikzpicture}

\bigskip

\begin{tikzpicture}[x=1.0cm, y=0.6cm, font=\small]
\def\rows{
 {$1$, $1$, $1$, $1$, $1$, $1$, , , ,$0$, $0$, $0$, $0$, $0$, $0$},
 {$1$, $4$, $1$, $2$, $2$, $2$, , , ,$3$, $3$, $1$, $3$, $3$, $1$},
 {$3$, $3$, $1$, $3$, $3$, $1$, ,$\xlongrightarrow[6:3]{p_3}$, ,$8$, $2$, $2$, $8$, $2$, $2$},
 {$2$, $2$, $1$, $4$, $1$, $2$, , , ,$3$, $1$, $3$, $3$, $1$, $3$},
 {$1$, $1$, $1$, $1$, $1$, $1$, , , ,$0$, $0$, $0$, $0$, $0$, $0$},
}
\foreach \row [count=\i from 0] in \rows {
\foreach \val [count=\j from 0] in \row {
\ifx\val\empty
\else
\node at ({\j + 0.5*\i}, {-\i}) {\val};
\fi
}}
\end{tikzpicture}

\bigskip

\begin{tikzpicture}[x=1.0cm, y=0.6cm, font=\small]
\def\rows{
 {$1$, $1$, $1$, $1$, $1$, $1$, , , ,$0$, $0$, $0$, $0$, $0$, $0$},
 {$3$, $1$, $3$, $1$, $3$, $3$, , , ,$2$, $2$, $2$, $2$, $2$, $2$},
 {$2$, $2$, $2$, $2$, $2$, $2$, ,$\xlongrightarrow[2:1]{p_3}$, ,$3$, $3$, $3$, $3$, $3$, $3$},
 {$3$, $1$, $3$, $1$, $3$, $1$, , , ,$2$, $2$, $2$, $2$, $2$, $2$},
 {$1$, $1$, $1$, $1$, $1$, $1$, , , ,$0$, $0$, $0$, $0$, $0$, $0$},
}
\foreach \row [count=\i from 0] in \rows {
\foreach \val [count=\j from 0] in \row {
\ifx\val\empty
\else
\node at ({\j + 0.5*\i}, {-\i}) {\val};
\fi
}}
\end{tikzpicture}

\begin{corollary}
The map $p_3:{\rm Frieze}(3) \to {\rm YFrieze}(3)$ is surjective.
\end{corollary}

\section{Arithmetic Y-frieze patterns of width $4$}

In similar methods to the case $n=3$, we determine all arithmetic Y-frieze patterns of width $4$.
The entries of the fundamental domain of the glide symmetry are given as follows.

\begin{center}
\begin{tikzpicture}[x=1.2cm, y=0.7cm, font=\small]
\def\rows{
  {$0$, $0$, $0$, $0$, $0$, $0$},
  {$a$, $e$, $i$, $l$, $n$},
  {$b$, $f$, $j$, $m$,  },
  {$c$, $g$, $k$,  ,  },
  {$d$, $h$,  ,  , },
  {$0$,  ,  , },
}
\foreach \row [count=\i from 0] in \rows {
  \foreach \val [count=\j from 0] in \row {
    \ifx\val\empty
    \else
      \node at ({\j + 0.5*\i}, {-\i}) {\val};
    \fi
  }
}
\end{tikzpicture}
\end{center}

The entries in the above arithmetic Y-frieze of width $4$ 
satisfy the following $10$ equations by the diamond relation
$WE=(N+1)(S+1)$:
\begin{align}\label{10equations}
&ae=1+b, \quad bf=(1+e)(1+c), \quad cg=(1+f)(1+d), \quad dh=1+g, \quad ei=1+f\\
&fj=(1+i)(1+g), \quad gk=(1+j)(1+h), \quad il=1+j, \quad jm=(1+l)(1+k), \quad ln=1+m.\\
\end{align}

Solving these equations by given $a$, $b$, $c$ and $d$, the other $10$ entries are represented as follows:
\begin{itemize}
\item $e = \frac{b + 1}{a}$,
\item $f = \frac{(c+1) (a+b+1)}{a b}$,
\item $g =\frac{(d+1) (a b+a c+b c+a+b+c+1)}{a b c}$,
\item $h =\frac{a b c+a b d+a c d+b c d+a b+a c+a d+b c+b d+c d+a+b+c+d+1}{a b c d}$,
\item $i = \frac{ab + ac  + bc + a + b + c + 1}{b(b+ 1)}$,
\item $j = \frac{(b+c+1) (a b c+a b d+a c d+b c d+a b+a c+a d+b c+b d+c d+a+b+c+d+1)}{b (b+1) c (c+1)}$,
\item $k = \frac{(a+1) (bc+bd+cd+b+c+d+1)}{bcd}$,
\item $l = \frac{bc + bd  + cd + b + c + d + 1}{c(c+1)}$,
\item $m = \frac{(b+1)(c+d+1)}{cd}$,
\item $n = \frac{c + 1}{d}$.
\end{itemize}

Since they are all positive integers, the numerator in each fraction is more than or equal to its denominator.
It follows that the following $10$ inequalities are satisfied:
\begin{itemize}
\item[(i)] ${b + 1} \geq {a}$,
\item[(ii)] $(c+1) (a+b+1) \geq {a b}$,
\item[(iii)] $(d+1) (ab+ac+bc+a+b+c+1) \geq a b c$,
\item[(iv)] ${abc + abd +  acd + bcd + ab + ac + ad + bc + bd + cd +a  + b  + c + d + 1} \geq {abcd}$,
\item[(v)] ${ab + ac + bc + a + b + c + 1} \geq {b(b+1)}$,
\item[(vi)] $(b+c+1) (a b c+a b d+a c d+b c d+a b+a c+a d+b c+b d+c d+a+b+c+d+1) \geq b (b+1) c (c+1)$,
\item[(vii)] $(a+1) (bc+bd+cd+b+c+d+1) \geq {bcd}$,
\item[(viii)] ${bc + bd + cd + b + c + d + 1} \geq {c(c+1)}$,
\item[(ix)] ${(b+1)(c+d+1)} \geq {cd}$,
\item[(x)] ${c + 1} \geq {d}$.
\end{itemize}

By using these 10 inequalities, we can show the following claims logically. 

\begin{proposition}\label{3.1}
If $a >5$ and $d > 5$, then inequality {\rm (iv)}
is not satisfied.
\end{proposition}

\begin{proof}
Inequality (iv) is equivalent to
$$
\frac{(a+1)(b+1)(c+1)(d+1)}{abcd} \geq 2.
$$
If $a > 5$ and $d > 5$, then $b \geq 5$ and $c \geq 5$ by (i) and (x).
In this case, the maximum of possible values of the left side term in the above inequality is
$49/25<2$. Hence, the inequality is not satisfied.
\end{proof}

By Proposition \ref{3.1}, we see that either $a \leq 5$ or $d \leq 5$ is satisfied.
Hereafter, we proceed our arguments under the condition $a \leq 5$.
Since inequalities (i)--(x) are symmetric with respect to $(a,b)$ and $(d,c)$,
we can obtain the symmetric conclusion if we start with the condition $d \leq 5$.
In this case, we have to handle inequality (iii) in the next proposition instead of (vii).

\begin{proposition}\label{3.2}
Suppose $a \leq 5$.
If $b >19$ and $d >19$, then inequality {\rm (vii)}
is not satisfied.
\end{proposition}

\begin{proof}
Under the condition $a \leq 5$, 
inequality (vii) implies
$$
5bc + 5bd + 5cd + 5b + 5c + 5d  + bc + bd + cd + 5 + b + c + d + 1 \geq {bcd},
$$
which is equivalent to
$$
\frac{(b+1)(c+1)(d+1)}{bcd} \geq \frac{7}{6}.
$$
If $b > 19$ and $d > 19$, then $c \geq 19$ by (x).
In this case, the maximum of possible values of the left side term in the above inequality is
$\frac{21 \cdot 21 \cdot 20}{20 \cdot 20 \cdot 19} <\frac{7}{6}$. Hence, the inequality is not satisfied.
\end{proof}

By Proposition \ref{3.2}, we see that either $b \leq 19$ or $d \leq 19$ is satisfied under the condition $a \leq 5$.
We first assume $b \leq 19$.

\begin{proposition}\label{3.3}
Suppose $b \leq 19$.
If $d >41$, then inequality {\rm (ix)}
is not satisfied.
\end{proposition}

\begin{proof}
Under the condition $b \leq 19$, 
inequality (ix) implies
$$
20c+20d+20 \geq cd,
$$
which is equivalent to
$$
\frac{(c+1)(d+1)}{cd} \geq \frac{21}{20}.
$$
If $d >41$, then $c \geq 41$ by (x).
In this case, the maximum of possible values of the left side term in the above inequality is
$\frac{43 \cdot 42}{42 \cdot 41} <\frac{21}{20}$. Hence, the inequality is not satisfied.
\end{proof}

Proposition \ref{3.3} implies that if either $b \leq 19$ or $d \leq 19$ then $d \leq 41$.
Thus, we have obtained that if $a \leq 5$ then $d \leq 41$.

\begin{proposition}\label{3.4}
Suppose $a \leq 5$ and $d \leq 41$.
If $b >102$ and $c >102$, then inequality {\rm (vi)}
is not satisfied.
\end{proposition}

\begin{proof}
Under the condition $a \leq 5$ and $d \leq 41$, 
inequality (vi) implies
$$
(b+c+1) (47bc+247b+247c+247) \geq b (b+1) c (c+1).
$$
This is not satisfied if $b >102$ and $c >102$.
\end{proof}

By Proposition \ref{3.4}, we see that either $b \leq 102$ or $c \leq 102$ is satisfied under the conditions
$a \leq 5$ and $d \leq 41$.

\begin{proposition}\label{3.5}
$(1)$
Suppose $d \leq 41$ and $b \leq 102$.
If $c >168$, then inequality {\rm (viii)}
is not satisfied. $(2)$ Suppose $d \leq 41$ and $c \leq 102$.
If $b >168$, then inequality {\rm (v)}
is not satisfied.
\end{proposition}

\begin{proof}
(1) Under the condition $d \leq 41$ and $b \leq 102$,
inequality (viii) implies
$$
c^2-143c-4326 \leq 0.
$$
This is not satisfied if $c >168$. (2) Under the condition $d \leq 41$ and $c \leq 102$,
inequality (v) implies
$$
b^2-143b-4326 \leq 0.
$$
This is not satisfied if $b >168$.
\end{proof}

Proposition \ref{3.5} implies that either $b \leq 102$ and $c \leq 168$, or $b \leq 168$ and $c \leq 102$. Therefore, if $a \leq 5$, then $d \leq 41$, and this condition is also satisfied. By the symmetry noted above, if $d \leq 5$, then $a \leq 41$, and the same conclusion follows. As a consequence, we obtain the following.

\begin{lemma}
An arithmetic Y-frieze of width $4$ exists only when either
\begin{itemize}
\item $a \leq 5$, $b \leq 102$, $c \leq 168$, and $d \leq 41$,
\item $a \leq 5$, $b \leq 168$, $c \leq 102$, and $d \leq 41$,
\item $a \leq 41$, $b \leq 102$, $c \leq 168$, and $d \leq 5$, or
\item $a \leq 41$, $b \leq 168$, $c \leq 102$, and $d \leq 5$.
\end{itemize}
\end{lemma}

In this setting, once bounds for $(a,b,c,d)$ are established, we can similarly identify 
the cases where $e,f,g,h,i,j,k,l,m,n$ are positive integers through numerical experiments—verifying the equations finitely many times—and thus obtain all integer solutions to \eqref{10equations}. 

\begin{theorem}\label{42(abcd)}
An arithmetic Y-frieze pattern of width $4$ appears exactly in the following 42 sets of
$(a,b,c,d,e,f,g,h,i,j,k,l,m,n)$:
\end{theorem}

\begin{tabular}{ll}
(1, 1, 2, 3, 2, 9, 20, 7, 5, 14, 6, 3, 2, 1) &
(1, 1, 4, 5, 2, 15, 24, 5, 8, 15, 4, 2, 1, 1)\\ 
(1, 2, 3, 2, 3, 8, 9, 5, 3, 5, 4, 2, 3, 2)&
(1, 2, 3, 4, 3, 8, 15, 4, 3, 8, 3, 3, 2, 1)\\ 
(1, 2, 6, 7, 3, 14, 20, 3, 5, 9, 2, 2, 1, 1)&
(1, 2, 9, 5, 3, 20, 14, 3, 7, 6, 2, 1, 1, 2)\\
(1, 3, 8, 3, 4, 15, 8, 3, 4, 3, 2, 1, 2, 3)&
(1, 4, 5, 3, 5, 9, 8, 3, 2, 3, 2, 2, 3, 2)\\ 
(1, 4, 15, 8, 5, 24, 15, 2, 5, 4, 1, 1, 1, 2)& 
(1, 6, 14, 5, 7, 20, 9, 2, 3, 2, 1, 1, 2, 3)\\ 
(2, 1, 1, 2, 1, 4, 15, 8, 5, 24, 15, 5, 4, 1)& 
(2, 1, 2, 3, 1, 6, 14, 5, 7, 20, 9, 3, 2, 1)\\ 
(2, 3, 2, 1, 2, 3, 4, 5, 2, 5, 9, 3, 8, 3)& 
(2, 3, 2, 3, 2, 3, 8, 3, 2, 9, 5, 5, 4, 1)\\ 
(2, 3, 4, 5, 2, 5, 9, 2, 3, 8, 3, 3, 2, 1)& 
(2, 3, 8, 3, 2, 9, 5, 2, 5, 4, 3, 1, 2, 3)\\
(2, 5, 9, 2, 3, 8, 3, 2, 3, 2, 3, 1, 4, 5)& 
(2, 9, 5, 2, 5, 4, 3, 2, 1, 2, 3, 3, 8, 3)\\ 
(2, 9, 20, 7, 5, 14, 6, 1, 3, 2, 1, 1, 2, 3)& 
(2, 15, 24, 5, 8, 15, 4, 1, 2, 1, 1, 1, 4, 5)\\ 
(3, 2, 1, 1, 1, 2, 6, 7, 3, 14, 20, 5, 9, 2)& 
(3, 2, 1, 2, 1, 2, 9, 5, 3, 20, 14, 7, 6, 1)\\ 
(3, 2, 2, 3, 1, 3, 8, 3, 4, 15, 8, 4, 3, 1)& 
(3, 2, 3, 2, 1, 4, 5, 3, 5, 9, 8, 2, 3, 2)\\ 
(3, 5, 4, 1, 2, 3, 2, 3, 2, 3, 8, 2, 9, 5)& 
(3, 8, 3, 1, 3, 2, 2, 3, 1, 3, 8, 4, 15, 4)\\ 
(3, 8, 3, 2, 3, 2, 3, 2, 1, 4, 5, 5, 9, 2)& 
(3, 8, 9, 5, 3, 5, 4, 1, 2, 3, 2, 2, 3, 2)\\ 
(3, 8, 15, 4, 3, 8, 3, 1, 3, 2, 2, 1, 3, 4)& 
(3, 14, 20, 3, 5, 9, 2, 1, 2, 1, 2, 1, 6, 7)\\ 
(3, 20, 14, 3, 7, 6, 2, 1, 1, 1, 2, 2, 9, 5)& 
(4, 3, 2, 1, 1, 2, 3, 4, 3, 8, 15, 3, 8, 3)\\ 
(4, 15, 8, 3, 4, 3, 2, 1, 1, 2, 3, 3, 8, 3)& 
(5, 4, 1, 1, 1, 1, 4, 5, 2, 15, 24, 8, 15, 2)\\ 
(5, 4, 3, 2, 1, 2, 3, 2, 3, 8, 9, 3, 5, 2)& 
(5, 9, 2, 1, 2, 1, 2, 3, 1, 6, 14, 7, 20, 3)\\ 
(5, 9, 8, 3, 2, 3, 2, 1, 2, 3, 4, 2, 5, 3)& 
(5, 14, 6, 1, 3, 2, 1, 2, 1, 2, 9, 3, 20, 7)\\ 
(5, 24, 15, 2, 5, 4, 1, 1, 1, 1, 4, 2, 15, 8)& 
(7, 6, 2, 1, 1, 1, 2, 3, 2, 9, 20, 5, 14, 3)\\ 
(7, 20, 9, 2, 3, 2, 1, 1, 1, 2, 6, 3, 14, 5)& 
(8, 15, 4, 1, 2, 1, 1, 2, 1, 4, 15, 5, 24, 5)\\
\end{tabular}
\medskip

\begin{corollary}
The map $p_4:{\rm Frieze}(4) \to {\rm YFrieze}(4)$ is bijective.
\end{corollary}

\begin{proof}
It is observed in \cite{dSG0} that the injectivity of the map $p_n : {\rm Frieze}(n) \to {\rm YFrieze}(n)$ for even $n \geq 1$ follows from \cite[Remark 1.18]{MG}. Since $|{\rm Frieze}(4)| = C_5 = 42$ and $|{\rm YFrieze}(4)| = 42$ by Theorem \ref{42(abcd)}, it follows that $p_4$ is bijective.
\end{proof}

\end{document}